\documentclass[10pt,a4paper]{amsart}

\usepackage[T1]{fontenc}
\usepackage[utf8]{inputenc} 
\usepackage{lmodern}        
\usepackage{microtype}      
\usepackage{amsmath,amssymb,amsfonts,amsthm,mathtools}
\usepackage{enumitem}
\usepackage{physics}
\usepackage[numbers]{natbib}

\usepackage{tikz}
\usetikzlibrary{arrows.meta, calc}
\usetikzlibrary{decorations.pathmorphing}

\usepackage{fancyref}
\usepackage[hidelinks]{hyperref}

\newtheorem{theorem}{Theorem}[section]

\newtheorem*{theorem*}{Theorem}

\newtheorem{lemma}[theorem]{Lemma}
\newtheorem{corollary}[theorem]{Corollary}

\newtheorem{conjecture}[theorem]{Conjecture}

\theoremstyle{definition}
\newtheorem{definition}[theorem]{Definition}

\newcommand{\R}{\mathbb{R}}
\newcommand{\Z}{\mathbb{Z}}

\newcommand{\Q}{\mathbb{Q}}

\newcommand{\namedref}[2]{\hyperref[#2]{#1~\ref*{#2}}}

\newcommand{\figref}[1]{\namedref{Figure}{#1}}

\newcommand{\thmref}[1]{\namedref{Theorem}{#1}}

\newcommand{\lemref}[1]{\namedref{Lemma}{#1}}
\newcommand{\corref}[1]{\namedref{Corollary}{#1}}
\newcommand{\conjref}[1]{\namedref{Conjecture}{#1}}
\newcommand{\defref}[1]{\namedref{Definition}{#1}}

\title[Failure of context-freeness of quasigeodesic languages]{Quasigeodesic languages are not context-free in some non-hyperbolic groups}
\author{Arya Saranathan}

\begin{document}

\begin{abstract}

    We study the full language of quasigeodesics in Cayley graphs with fixed error constants. We show that, given a non-virtually-cyclic nilpotent group or Baumslag--Solitar group, and any finite generating set, such languages fail to be context-free for sufficiently large error constants. In fact, this conclusion holds for any finitely generated group which contains one of these groups as an undistorted subgroup. This strengthens a recent theorem of Hughes, Nairne, and Spriano, who showed that such languages fail to be regular in any non-hyperbolic group, for sufficiently large error constants.
\end{abstract}
\maketitle

\section{Introduction}
Given a group $G$ and finite generating set $S$, we can form its Cayley graph $\Gamma(G,S)$, which has a vertex for every group element, and edges labelled by generators. 
In this paper, we are interested in the quasigeodesics of $\Gamma(G,S)$, which are shortest paths in $\Gamma(G,S)$ up to some fixed additive and multiplicative error. In particular, since edges are labelled by generators from $S$, every path corresponds to a word in the generators $S$, and we consider the formal language of words described by the quasigeodesics (with fixed error constants). Geodesic and quasigeodesic languages have been studied extensively. For example, 
Groves showed that there is no regular set of geodesic normal forms 
for words in $BS(1,n)$ \cite{groves}, and Elder has shown that the 
geodesic language of $BS(1,2)$ in the standard generating set is not a counter language \cite{ElderCounterBS}.
Moreover, it was shown by Cleary, Elder, and Taback that the full language of geodesics of the lamplighter group 
is context-free, with respect to the standard wreath-product generating set \cite{ClearyElderTaback}. It happens that the algorithmic properties of these quasigeodesic languages are closely connected to the geometry of the group in a striking way. The bridge is the fundamentally geometric notion of a hyperbolic group, originally introduced by Gromov in \cite{Gromov1987}; a recurring theme is that the class of hyperbolic groups seems to draw a boundary between `tame' groups with simple quasigeodesic languages, and `wild' groups with complex quasigeodesic languages. To make this precise, the following definition was introduced in \cite{HoltRees2003} and is of particular interest to this paper.

\begin{definition}\label{def:qreg}
    A finitely generated group $G$ is $\Q$\textbf{REG} if for all rational
    $\lambda \ge 1$, real $\epsilon \ge 0$ and finite generating sets $S$, the $(\lambda, \epsilon)$-quasigeodesics in
    the Cayley graph $\Gamma(G, S)$ form a regular language.
\end{definition} 

Holt and Rees showed in \cite{HoltRees2003} that hyperbolic groups are $\Q$\textbf{REG}. Some intuition for this result comes from the Morse lemma for hyperbolic groups. Quasigeodesics in the Cayley graph between two elements of a hyperbolic group fellow-travel, so an automaton recognizing the quasigeodesic language need only check that an input word does not deviate from a given quasigeodesic by too much.
Remarkably, Hughes, Nairne, and Spriano proved the converse in \cite{Hughes_2025},
establishing that the regularity of rational quasigeodesic languages
\textit{characterizes} hyperbolicity.

\begin{theorem}[\cite{Hughes_2025, HoltRees2003}]\label{thm:HNS}
    A finitely generated group is hyperbolic if and only if it is $\Q$\emph{\textbf{REG}}.
\end{theorem}
 
 \thmref{thm:HNS}  lends credence to the idea that hyperbolic groups form a very natural class that can be defined in many different ways. Indeed, the result adds to 
a growing list of diverse characterizations of hyperbolicity: 
thin triangles condition \cite{Gromov1987},
vanishing of $\ell^{\infty}$-cohomology \cite{Gersten1998}, 
all asymptotic cones being $\R$-trees \cite{Gromov1987},
and as strongly geodesically automatic groups \cite{Papasoglu1995}.

In this paper we explore the question of whether a stronger characterization is possible in the same vein. Define the 
analogous property $\Q$\textbf{CF} by replacing `regular' with `context-free' in \defref{def:qreg}. While context-freeness is weaker than regularity, it still places
limits on the memory available to an automaton recognizing the quasigeodesic languages, and would thus impose constraints on the geometry of the group. 
This suggests the following audacious conjecture.

\begin{conjecture}\label{conj:main}
    A finitely generated group is hyperbolic if and only if it is $\Q$\emph{\textbf{CF}}.
\end{conjecture}

Such a strengthening of \thmref{thm:HNS} would be an enthralling surprise, as a linguistic gap would appear: \conjref{conj:main} implies that a finitely generated group that is $\Q$\textbf{CF} is actually $\Q$\textbf{REG}. 
This would be an interesting parallel to the behaviour of the Dehn function, where a similar gap appears: hyperbolic groups 
are precisely those finitely presented groups with linear Dehn function, and any group with subquadratic Dehn function must
be hyperbolic and therefore have linear Dehn function \cite{Olshanskii1991Subquadratic}. 
To show the strength of the conjecture we remark that the behaviour of context-free languages is really not well understood under a change of generating sets. Indeed, even the following significantly weaker statement is currently unknown: if a group admits a finite generating set for which rational quasigeodesic languages in the Cayley graph are context-free, must rational quasigeodesic languages be context-free for any choice of finite generating set? Nevertheless, motivated by these considerations, in this paper we verify \conjref{conj:main} for various classes of groups. First, we prove the following.

\begin{theorem}\label{thm:polygrowth}
    A finitely generated nilpotent group is $\Q$\emph{\textbf{CF}} if and only if it is virtually cyclic.
\end{theorem}

This fits with \conjref{conj:main} because the hyperbolic nilpotent groups are exactly the virtually cyclic ones. Secondly, we also verify \conjref{conj:main} for Baumslag--Solitar groups, which are \emph{not} hyperbolic. 

\begin{theorem}
    Baumslag--Solitar groups $BS(m,n)$ are not $\Q$\emph{\textbf{CF}}. 
\end{theorem}

In this paper, we do not work directly with the definition of $\Q$\textbf{CF}, as it poses some challenges. 
For example, as mentioned, \emph{a priori} 
there could be a group which has context-free rational quasigeodesic languages for one choice of generating set, but not another. It is also not clear how this property interacts with undistorted subgroups.
To remedy this, we opt to introduce the notion of a \emph{chaotic} group, a direct strengthening of the not-$\Q$\textbf{CF} property. This reframing is important to the proofs of the above theorems, and it ultimately allows us to lift to overgroups in which these groups sit as undistorted subgroups,
significantly expanding the class of groups for which we can 
confirm \conjref{conj:main}. Our key findings can thus be summarised as follows. 

\begin{theorem}
    Let $G$ be a finitely generated group which contains an undistorted 
    subgroup which is either non-virtually-cyclic and nilpotent, or a Baumslag--Solitar group. 
    Then $G$ is chaotic, and in particular not $\Q$\emph{\textbf{CF}}.
\end{theorem}

An important tool used to obtain our results is Ogden's lemma from \cite{ogden68}, which provides a strong necessary condition for a language to be context free, and hence a way to prove that a given language is not context free. Then, the general proof strategy is to construct a family of words in a $(\lambda, \epsilon)$-quasigeodesic language of a group, and to argue that this family of words has properties violating the necessary condition given by Ogden's lemma, showing that the $(\lambda, \epsilon)$-quasigeodesic language is not context-free. Along the way we rely on Osin's theorem on subgroup distortion in nilpotent groups \cite{osin}, and later use results by Burillo and Elder \cite{BurilloElder2014} on normal forms and lower bounds on displacement for words in Baumslag–Solitar groups.

\subsection*{Outline}
In \S2 we set up notation and establish a useful corollary of Ogden's lemma, which will be used several times to prove that certain languages are not context-free.
In \S3 we introduce the notion of a \textit{chaotic} group, 
which is a stronger property than not-$\Q$\textbf{CF}, and
is easier to work with. In \S4 we prove that $\Z^2$ is chaotic. We then extend 
the methods of \S4 to show that non-virtually-cyclic virtually nilpotent groups are chaotic, 
in \S5. Finally, in \S6 we use a slightly more involved construction to
show that Baumslag--Solitar groups are chaotic.

\subsection*{Acknowledgments}
This was completed as part of a summer project supervised by 
Joseph Paul MacManus (main supervisor) and Davide Spriano. I am very grateful
for the weekly discussions and 
email exchanges that were crucial to guiding my progress and 
that made this work possible. 
Their role in suggesting problems, elucidating the relevant literature, 
and clarifying my thinking was invaluable. I would like to thank them also for their careful reading of drafts, and feedback that helped improve this paper. The project
was supported by a summer research grant from 
the Mathematical Institute, Oxford.

\section{Preliminaries}

\subsection{Notation and terminology}
Consider a group $G$ with finite generating set $S$. 
In this paper we 
deal exclusively with finitely generated groups and finite generating sets, 
which are also assumed symmetric (closed under inverses).
We may then form the 
\emph{Cayley graph} $\Gamma(G,S)$, which is a 
geometric picture of the group's combinatorial 
structure. It has vertex set $G$, with directed edges $\{g, gs\}$
for all $g\in G$ and $s\in S$. By considering edges to all have unit length, distances in $\Gamma(G,S)$ 
induce a metric $d_S$ on $G$, and in this way $G$ is a discrete metric space.
Left-multiplication by an element of $G$ is an isometry of $\Gamma(G,S)$, 
so $d_S(h,g) = d_S(e, h^{-1}g)$ (here $e\in G$ is the identity element). Therefore, it will be convenient
to treat $d_S$ as a unary operator, and we will write $d_S(g)$ to mean $d_S(e, g)$.
When the generating set $S$ is clear we may write $d = d_S$.

A \emph{word} $w$ in $S$ is an element of $S^*$, i.e. a finite sequence of generators, 
whose length is denoted by $|w|$. We will frequently abuse notation 
and use $w$ to refer also to the corresponding element of $G$ (yielded by 
multiplying together its letters in $G$). A word is understood to trace a discrete path in 
the Cayley graph $\Gamma(G,S)$ starting from the identity $e\in G$, 
which gives rise to the following geometric notions.
A word $w$ is \emph{geodesic} if it describes a shortest path 
in $\Gamma(G,S)$, or equivalently $|w| = d(w)$. We will be 
especially interested in a coarse version of this property. The word 
$w$ is \emph{$(\lambda, \epsilon)$-quasigeodesic} if for any subword $u$ of $w$, 
we have 
\[ |u| \le \lambda d(u) + \epsilon .\]
By `subword' we mean any contiguous subsequence of letters of $w$.
Geodesics and quasigeodesics can be defined in much greater generality, 
but this discrete viewpoint will suffice for our purposes.
The collection of all geodesic words for $(G,S)$ is a subset of $S^*$ 
known as the \emph{geodesic language}, and the collection of all 
$(\lambda, \epsilon)$-quasigeodesics is the \emph{$(\lambda, \epsilon)$-quasigeodesic language}.
Note that these languages depend on the generating set $S$.

We will also need the notion of a quasiisometrically embedded (QI-embedded, or undistorted) subgroup. 
The subgroup $H \le G$ is QI-embedded roughly when distances in $H$ are distorted 
only by a uniform constant compared to distances in $G$. That is, if $T, S$ are generating sets for $H,G$ respectively, 
then we require constants $(A, B)$ such that 
for all $h\in H$,
\[  \tfrac{1}{A}d_T(h)-B \le d_S(h)  \le  Ad_T(h)+B .\]
However, it is straightforward to see that 
changing generating sets only affects
distances by a uniform constant, so the property of QI-embeddedness 
is independent of generating sets. 

\subsection{Context-free languages}

A language $L \subseteq S^*$ is context-free (or is a CFL) if it is the 
language recognized by a nondeterministic pushdown automaton, 
or equivalently if it is the language of a context-free grammar 
(see any textbook on formal languages, for example \cite{SipserTOC3}).

A particularly useful necessary condition for a language to be context-free 
is given by the following theorem of Ogden from \cite{ogden68}.

\begin{theorem}[Ogden's lemma]\label{thm:ogden}
    Let $L$ be a context-free language. Then there exists a constant $p$ 
    such that for any word $w \in L$ with at least $p$ marked positions, 
    there is a decomposition $w = uxzyv$ such that:
    \begin{itemize}
        \item $xzy$ contains at most $p$ marked positions,
        \item $z$ contains at least one marked position,
        \item $u,x$ both contain marked positions, or $y,v$ both contain marked positions, and
        \item for all $n \ge 0$, the word $ux^nzy^nv \in L$.
    \end{itemize}
\end{theorem}

We call this $p$ the Ogden constant of $L$.
We will make use of this theorem several times to prove that certain languages are 
not context-free. We record a corollary tailored to our use.

\begin{corollary}\label{cor:ogden}
    Let $L$ be context free with Ogden constant $p$, and 
    $\alpha \beta \gamma \in L$ with $|\beta| \ge p$.
    Then there is a decomposition $\alpha \beta \gamma = uxzyv$ such that:
    \begin{itemize}
        \item At least one of $x$ or $y$ is contained within $\beta$, and
        \item for all $n \ge 0$, the word $ux^nzy^nv \in L$.
    \end{itemize}
\end{corollary}
\begin{proof}
  Apply Ogden's lemma to $\alpha \beta \gamma$, marking all positions of $\beta$.
  There is a decomposition $\alpha \beta \gamma = uxzyv$ such that either both $u,x$ or both $y,v$ contain marked positions.
  Without loss of generality, suppose $u,x$ contain marked positions. Then $\alpha$ must be contained entirely in $u$. But $z$ also must 
  contain a marked position, so $ux$ is contained in $\alpha \beta $. It follows that $x$ is contained in $\beta$.
\end{proof}

\section{Chaotic groups}

The property of a group $G$ with generating set $S$ having non-context-free 
quasigeodesic languages is not necessarily well-behaved. For example, it may \textit{a priori} be 
sensitive to the choice of generating set, and may not be inherited by overgroups in which 
$G$ is QI-embedded. This makes it hard to prove not-$\Q$\textbf{CF} type results which generalize to large classes of groups. We therefore introduce a \emph{stronger} version of the not-$\Q$\textbf{CF} property that is 
stable with respect to the above operations.

\begin{definition}[Chaotic group]
A group $G$ with finite generating set $S$ is \textit{chaotic} if there 
exist $\lambda_0\ge 1$ and $\epsilon_0\ge 0$ such that 
for every context--free language 
$L\subseteq S^*$ that contains all $(\lambda_0,\epsilon_0)$--quasigeodesic 
words of $\Gamma(G,S)$, the following holds: for all $\lambda\ge 1$ and 
$\epsilon\ge 0$, the language $L$ contains a word that 
is not $(\lambda,\epsilon)$--quasigeodesic in $\Gamma(G,S)$. 
Then we say $(G,S)$ is chaotic, or $(G,S)$ is chaotic at $(\lambda_0,\epsilon_0)$. 
\end{definition}

This definition seems to depend both on $G$ and $S$. 
But as we shall see, the chaotic nature of a group is preserved 
when changing generating sets, so we may speak of a group $G$ itself being 
chaotic. We emphasize that if $G$ is chaotic, then it is not $\Q$\textbf{CF}.

Let $H\le G$ be finitely generated with 
finite generating sets $T$ for $H$ and $S$ for $G$. 
Assume $H\hookrightarrow G$ is a $(A,B)$--quasiisometric embedding: 
for all $h\in H$,
\begin{equation}\label{eq:QI}
 \tfrac{1}{A}d_T(h)-B \le d_S(h)  \le  Ad_T(h)+B.
\end{equation}
For each $t\in T$, choose a word $w_t\in S^*$, geodesic in $\Gamma(G,S)$,
representing $t$ in $H\le G$, and define 
the monoid homomorphism $h:T^*\to S^*$ by extending $h(t)=w_t$ for $t\in T$. 
Let $\ell_{\max}:=\max_{t\in T}|w_t|$ and 
write 
$R:=(w_{t_1}\mid\cdots\mid w_{t_m})^*\subseteq S^*$ 
where $T=\{t_1,\dots,t_m\}$.

\begin{lemma}[Block replacement preserves quasigeodesicity]\label{lem:block}
$ $
\begin{enumerate}[label=\rm(\alph*)]
 \item  If $w\in T^*$ labels a $(\lambda,\epsilon)$-quasigeodesic 
 in $\Gamma(H,T)$, then $h(w)$ labels a 
 $(\Lambda,\mathsf E)$--quasigeodesic in $\Gamma(G,S)$ with
 \begin{equation}\label{eq:forward}
    \Lambda \;=\; A\,\lambda \, \ell_{\max},\qquad \mathsf E \;=\; 2(\Lambda + 1)\ell_{\max} + \Lambda B + \ell_{\max}\epsilon.
 \end{equation}
 
 \item If $h(w)$ labels a $(\Lambda,\mathsf E)$--quasigeodesic in $\Gamma(G,S)$, then $w$ labels a $(\lambda,\epsilon)$--quasigeodesic in $\Gamma(H,T)$ with
 \begin{equation}\label{eq:reverse}
 \lambda \;=\; A\,\Lambda,\qquad \epsilon \;=\;  \Lambda \, B  + \mathsf E.
 \end{equation}
\end{enumerate}
\end{lemma}

\begin{proof}
  To prove (a) we need to verify the $(\Lambda,\mathsf E)$-quasigeodesic inequality for all subwords of $h(w)$.
  Subwords of $h(w)$ can be decomposed into blocks $w_{t_i}$, up to error at the 
  start and end. Concretely, any subword $v$ of $h(w)$ is of the form $v = ah(u)b$, where $u$ is a subword of $w$,
  $a$ is a proper suffix of some $w_{t_i}$, and $b$ is a proper prefix of some $w_{t_j}$.

  We have $d_S(h(u)) \ge \tfrac{1}{A}d_T(u) - B$
  from \eqref{eq:QI}, since $h(u)$ represents the same element of $H$ as $u$.
  Also, $|h(u)| \le |u|\ell_{\max}$ since each of the $|u|$ blocks in $h(u)$ has 
  length at most $\ell_{\max}$. Further, since $w$ is $(\lambda,\epsilon)$--quasigeodesic,
  and $u$ is a subword of $w$,
  we have $|u| \le \lambda d_T(u) + \epsilon$.  
  We account for the terms $a,b$ via estimates $|v| \le |h(u)| + 2\ell_{\max}$ 
  and $d_S(v) \ge d_S(h(u)) - 2\ell_{\max}$.
  Combining these shows that 
  \begin{align*} |v| &\le |h(u)| + 2\ell_{\max}  \\ &\le |u|\ell_{\max} + 2\ell_{\max} \\ 
  &\le (\lambda d_T(u) +\epsilon)\ell_{\max} + 2\ell_{\max} \\
  &= A\lambda \ell_{\max} \left( \frac{1}{A}d_T(u) - B - 2\ell_{\max}\right) + 2(\Lambda + 1)\ell_{\max} + \Lambda B + \ell_{\max}\epsilon  \\
  &\le \Lambda (d_S(h(u))-2\ell_{\max} ) + \mathsf E \\
  &\le \Lambda d_S(v) + \mathsf E.
  \end{align*}

  The proof of (b) is similar. Let 
  $u$ be a subword of $w$. Then $h(u)$ is a subword of $h(w)$, so we have 
  $|h(u)| \le \Lambda d_S(h(u)) + \mathsf E$. From \eqref{eq:QI} we 
  have $d_S(h)  \le  Ad_T(h)+B$. Lastly $|u| \le |h(u)|$ since all 
  blocks $w_{t_i}$ have positive length. Combining these gives
  \begin{align*}
    |u| &\le |h(u)| \le \Lambda d_S(h(u)) + \mathsf E \\
    &\le \Lambda (A d_T(u) + B) + \mathsf E \\
    &= A\Lambda d_T(u) + \Lambda B + \mathsf E \\
    &= \lambda d_T(u) + \epsilon,
  \end{align*}
  which shows that $w$ is $(\lambda, \varepsilon)$-quasigeodesic.
\end{proof}

\begin{theorem}[Inheriting from an undistorted subgroup]\label{thm:inherit}
    If $(H,T)$ is chaotic, where $T$ is some finite generating set, then $(G,S)$ is chaotic for all finite 
    generating sets $S$ of $G$.
\end{theorem}
\begin{proof}
    Let $(H,T)$ be chaotic at $(\lambda_0, \epsilon_0)$. Choose 
    constants $(\Lambda_0,\mathsf E_0)$ according to \eqref{eq:forward} with 
    $(\lambda_0, \epsilon_0)$ in place of $(\lambda, \epsilon)$, 
    and let $L\subset S^*$ be some CFL containing all $(\Lambda_0,\mathsf E_0)$-quasigeodesics in $\Gamma(G,S)$.
    Define $L':=h^{-1}(L\cap R)\subseteq T^*$.
    CFLs are closed under intersection with regular languages and inverse homomorphisms (see for example \cite{SipserTOC3}), so $L'$ is a CFL.
    By construction, $L'$ contains all $(\lambda_0,\epsilon_0)$-quasigeodesics in $\Gamma(H,T)$.
    Let $\Lambda\ge 1, \mathsf E \ge 0$ be arbitrary. Using \eqref{eq:reverse}, 
    choose $(\lambda, \epsilon)$ so that $(\Lambda, \mathsf E)$-quasigeodesics in $\Gamma(G,S)$
    are sent to $(\lambda, \epsilon)$-quasigeodesics in $\Gamma(H,T)$ by $h^{-1}$.
    Since $(H,T)$ is chaotic at $(\lambda_0, \epsilon_0)$, 
    $L'$ contains a word $w \in T^*$ that is not 
    $(\lambda,\epsilon)$-quasigeodesic
    in $\Gamma(H,T)$. Then $h(w) \in L$ is not $(\Lambda, \mathsf E)$-quasigeodesic in $\Gamma(G,S)$, 
    by the contrapositive of the reverse direction of \lemref{lem:block}. Since $\Lambda, \mathsf E$ were arbitrary, $(G,S)$ is chaotic at $(\Lambda_0, \mathsf E_0)$.
\end{proof}

\begin{corollary}\label{cor:gensets}
    If $(H,T)$ is chaotic, then so is $(H,T')$ for any other 
    finite generating set $T'$.
\end{corollary}
\begin{proof}
    Clearly $H$ is QI-embedded in itself. Then this follows immediately from \thmref{thm:inherit}.
\end{proof}

This confirms that we may speak of a group $G$ itself being chaotic,
and also that if $G$ is chaotic, then it is \textit{not} $\Q$\textbf{CF}.

\section{Abelian groups}
We will show that $\Z^2$ is chaotic, which by \thmref{thm:inherit} will then extend to all
groups containing an undistorted $\Z^2$. For example, 
it will apply to abelian groups of free rank at least $2$, which are precisely the non-virtually-cyclic abelian groups. The results of this section are all subsumed by the more general theorem (for groups containing an undistorted non-virtually-cyclic nilpotent group) which follows in \S5. However, the $\Z^2$ case provides a good illustration of the basic proof mechanism.

\begin{theorem}\label{thm:z2}
    $\Z^2$ is chaotic.
\end{theorem}
\begin{proof}
    We choose to work with the standard generating set $\{a,b\}$ where $a = (1,0)$ and $b = (0,1)$.
    For this generating set, we show that $\Z^2$ is chaotic at $(5,0)$. Let $L$ be a 
    CFL containing all $(5,0)$-quasigeodesic words in $\Z^2$, with Ogden constant $p$.
    Let $\lambda \ge 1$, $\epsilon \ge 0$ be arbitrary.
    Let $q = \lceil\max(p, \epsilon)\rceil$ and consider the word 
    \[ w = (b^{-q}a^{2q}b^{q})a^{-q}(b^qa^{2q}b^{-q}) . \]
    It is straightforward to check that $w$ is $(5,0)$-quasigeodesic (see \figref{fig:spiral}), so $w\in L$.
    We apply \corref{cor:ogden}, with $\alpha = (b^{-q}a^{2q}b^{q})$, $\beta = a^{-q}$, and $\gamma = (b^qa^{2q}b^{-q})$. 
    This gives us a decomposition $w = uxzyv$, where either $x$ or $y$ is fully contained in 
    $\beta = a^{-q}$. Without loss of generality suppose $x$ is fully contained in $a^{-q}$, so that $x = a^{-k}$ for some $1 \le k \le q$.
    Defining $w_n:= ux^nzy^nv \in L$, we then have
    \[ w_n = (b^{-q}a^{2q}b^{q})a^{-q -kn}X_n \] 
    for some words $X_n$; importantly, the prefix $(b^{-q}a^{2q}b^{q})$ is not affected by the pumping. Choose 
    $n \ge q/k$, so that $w_n$ has a prefix $w' = b^{-q}a^{2q}b^{q}a^{-2q}$. But this subword 
    is not $(\lambda, \epsilon)$-quasigeodesic, because $|w'| = 6q$ and $d(w') = 0$, so in particular 
    $|w'| \ge 6\epsilon > \lambda \cdot 0 + \epsilon$. Since $w_n \in L$, we have shown that 
    $L$ contains a word that is not $(\lambda, \epsilon)$-quasigeodesic.
\end{proof}

\begin{figure}[h]
    \begin{tikzpicture}[
    x=2cm,y=2cm,
    >=Latex,
    semithick,
    line cap=round,line join=round,
    baseline={(current bounding box.center)}
    ]
    \tikzset{
        seglabel/.style={pos=0.5,font=\scriptsize,text=gray!70,inner sep=0.6pt},
        hintarrow/.style={dashed,-{Latex},shorten <=2pt,shorten >=2pt},
        >={Latex[length=2.2mm,width=1.6mm]}
    }

    \draw (0,0)
        -- ++(0,-1) node[seglabel,left=2pt]  {$b^{-q}$}
        -- ++(2,0)  node[seglabel,below=2pt] {$a^{2q}$}
        -- ++(0,1)  node[seglabel,right=2pt] {$b^{q}$}
        -- ++(-1,0) node[seglabel,above=2pt] {$a^{-q}$}
        -- ++(0,1)  node[seglabel,right=2pt] {$b^{q}$}
        -- ++(2,0)  node[seglabel,above=2pt] {$a^{2q}$}
        -- ++(0,-1) node[seglabel,right=2pt] {$b^{-q}$};

    \fill (0,0) circle (1pt) node[above left] {start};
    \fill (3,0) circle (1pt) node[below right] {end};

    \draw[hintarrow] (1,0) -- (0,0); 
    \draw[hintarrow] (2,0) -- (3,0); 
    \end{tikzpicture}
    \caption{The word $w$, with pumping along dashed lines}
    \label{fig:spiral}
\end{figure}
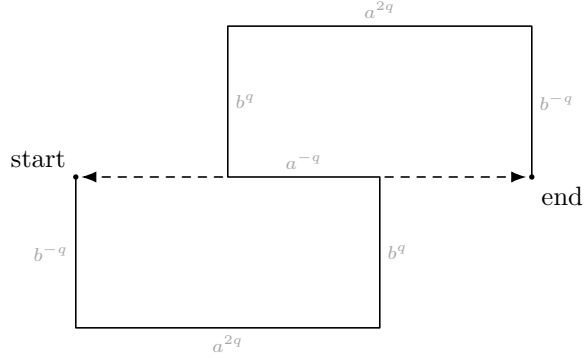

\begin{corollary}
    Non-virtually-cyclic finitely generated abelian groups are chaotic.
\end{corollary}
\begin{proof}
    Let $G$ be non-virtually-cyclic, finitely generated, and abelian. Then $G$ has free rank at least $2$, and it follows that $G$ contains a QI-embedded $\Z^2$. 
    Now the result follows by \thmref{thm:inherit}.
\end{proof}

The next technical lemma allows us to be more flexible with the generating set, which will be useful later.

\begin{lemma}\label{lem:tech}
    Let $S$ be a finite generating set for $\Z^2$, 
    and $a, b\in S$ be any two independent generators.
    For some $\lambda\ge 1$, and for $q\ge 1$,
    the words $w_q = (b^{-q}a^{2q}b^{q})a^{-q}(b^qa^{2q}b^{-q})$
    are uniformly $(\lambda,0)$-quasigeodesic in $\Gamma(\Z^2, S)$. 
    
\end{lemma}
\begin{proof}
    The proof of \thmref{thm:z2} almost carries over, since the specific choice of $a = (1,0)$ and $b = (0,1)$ made there is not strictly necessary. The only additional complication is that the additional generators in $S$ 
    apart from $a,b$ may allow for shorter paths, thereby reducing the displacement 
    in $d_S$ norm of subwords. However, since $S$ is finite, there is a constant $C$ depending on $S$ such that
    for all $x\in \Z^2$, $d_S(x) \ge Cd_{\{a,b\}}(x)$. In particular, since 
    the $w_q$ are $(5,0)$-quasigeodesic in $\Gamma(\Z^2, \{a,b\})$ (as in \thmref{thm:z2}),
    the $w_q$ are $(5C,0)$-quasigeodesic in $\Gamma(\Z^2, S)$. So we can take $\lambda = 5C$.
\end{proof}

\section{Nilpotent groups}

The goal is to prove that non-virtually-cyclic virtually 
nilpotent groups are chaotic. This class of groups 
is much bigger than the class for which the results of \S4 apply: many nilpotent groups contain no undistorted $\Z^2$.
For example,
the discrete Heisenberg group of 
matrices of the form $\begin{psmallmatrix}
1 & x & z \\
0 & 1 & y \\
0 & 0 & 1
\end{psmallmatrix}$ for 
integer $x,y,z$ is nilpotent but nonabelian, and contains no undistorted $\Z^2$. Our basic tool will be similar: 
a spiral word that can be pumped using Ogden's lemma 
to give arbitrarily `bad' quasigeodesics. We will 
however need some results on the structure and metric distortion of nilpotent groups.

We will use the convention $[a,b] = a^{-1}b^{-1}ab$. If $G$ is a group, 
and $A, B \subseteq G$ are subsets, then $[A,B]$ denotes the subgroup generated 
by $\{[a,b]: a\in A, b \in B \}$. 
Let $G$ be a group, set $G_1 = G$, 
and recursively define $G_{i+1} = [G_i, G]$. Then we have
$G = G_1 \triangleright G_2 \triangleright \cdots$
which is the \textit{lower central series} for $G$. 
If eventually the $G_i$ are all trivial, $G$ is \textit{nilpotent}. 
In that case, the \textit{nilpotency class} of $G$ is the least $c$ such that 
$G_{c+1}$ is trivial.
The following is a standard fact: 
if $H\le G$ is a finite index subgroup, then $H$ is QI-embedded in $G$.
This implies that it suffices to consider nilpotent groups, and chaoticity 
will immediately extend to virtually nilpotent groups.
Denote $G' = [G,G]$ and write $G_{\text{ab}} = G/G'$ for the abelianization of $G$.
The next technical lemma arises from a MathOverflow post of Yves Cornulier \cite{Cornulier:MO:Abelianization:2020}, 
and is proven in more detail in \cite{auffinger2021asymptoticshapesstationarypassage}; 
see Appendix B, Lemma 7 and discussion just preceding it.
\begin{lemma}\label{lem:cut}
    Let $G$ be nilpotent, where $G_{\text{ab}} \cong T \oplus \Z^r$ for 
    some torsion abelian group $T$ and $r \ge 0$.
    Then there is a nilpotent, finite-index subgroup $H\le G$ with 
    $H_{\text{ab}} \cong \Z^r$.
\end{lemma}

\begin{theorem}
    If $G$ is a nilpotent group, either it is virtually cyclic, 
    or the abelianization $G_{\text{ab}}$ has free rank at least $2$.
\end{theorem}
\begin{proof}
    If $G_{\text{ab}}$ is torsion, then by \lemref{lem:cut} there is a finite-index nilpotent subgroup $H\le G$
    with trivial abelianization, i.e. $G$ is virtually trivial and hence finite.
    If $G_{\text{ab}}$ has free rank $1$, then $G_{\text{ab}} = T \oplus \Z$ for some torsion group $T$.
    Then \lemref{lem:cut} gives a finite-index nilpotent subgroup $H\le G$ with 
    $H_{\text{ab}} \cong \Z$. We will show that $H$ is abelian. 
    Consider $K = H/H_3$, which is nilpotent of class at most $2$. 
    It is easy to check that $K_{\text{ab}} \cong H_{\text{ab}} \cong \Z$, 
    and since $K$ is class at most $2$, we have $K' \le Z(K)$,
    Thus $K$ is a central extension of $K'$ by $\Z$, 
    so in fact $K$ is abelian. Since $H/H'$ is the largest abelian quotient of $H$, 
    this implies $H' = H_2 \le H_3$, and so $H_2 = H_3 = \dots$. Since 
    $H$ is nilpotent and the lower central series must terminate, 
    this shows $H' = H_2$ is trivial, so $H \cong H_{\text{ab}} \cong \Z$. Then 
    $H$ is virtually $\Z$.
    In other cases $G_{\text{ab}}$ has free rank at least $2$, so we are done.
\end{proof}

So from now on we assume that $G$ is a finitely generated nilpotent group
where $G_{\text{ab}}$ has free rank at least $2$.
Let $\pi: G \to G_{\text{ab}}$ be the abelianization map
and write $\pi(x) = \overline x$ for $x\in G$.
Since $G_{\text{ab}}$ has free rank at least $2$,
we may choose $a,b\in G \setminus G'$ such that $\overline a, \overline b$
generate a rank-2 free abelian subgroup of $G_{\text{ab}}$, that is a $\Z^2$ in $G_{\text{ab}}$.
We work with some arbitrary finite generating set $S$ that contains $a$ and $b$, and treat $S$ as fixed.
Define the family of words $w_n = (b^{-n}a^{n^2}b^n)a^{-n}(b^na^{n^2}b^{-n})$. Note the similarity to the words used in the proof of \thmref{thm:z2}; the idea is that in the Cayley graph of the abelianization $G_{\text{ab}}$, the words $w_n$ look very similar to the spiral in \figref{fig:spiral}. Therefore the proof of \thmref{thm:z2} can of course show that $G_{\text{ab}}$ is chaotic. But we will argue that upon projecting back up to the nilpotent group $G$, the distances of subwords of $w_n$ cannot change too much, and here we will leverage important facts about nilpotent groups. As a result, we can show that even in $G$, the $w_n$ are quasigeodesic for some suitable constants, and the argument using Ogden's lemma goes through. The use of $n$ and $n^2$ is an important detail, as will become clear.

The proof now proceeds in two parts. First, we show that the $w_n$ 
are uniformly quasigeodesic in $G$. Then we pump using Ogden's lemma, 
and show that the resulting word can be made arbitrarily inefficient.
This will then show that $G$ is chaotic.

\begin{lemma}\label{lem:uniform}
    There exists $\lambda_0$ such that $w_n$ is $(\lambda_0,0)$-quasigeodesic for all $n\ge 2$.
\end{lemma}
\begin{proof}
    Pass to the abelianization $G_{\text{ab}}$, with generating set 
    $\overline S = \{ \overline s \mid s\in S \}$.
    Then clearly $d_{\overline S}(\overline z) \le d_S(z)$ for all $z\in G$, 
    because if $z = s_1 s_2 \cdots s_k$ then $\overline z = \overline s_1 \overline s_2 \cdots \overline s_k$.
    We know $\langle \overline a, \overline b \rangle$ is a subgroup isomorphic to $\Z^2$ 
    in the abelian group $G_{\text{ab}}$. It follows from \lemref{lem:tech} that there exists
    $\lambda_0$ such that 
    $w_n$ is $(\lambda_0,0)$-quasigeodesic in $G_{\text{ab}}$ for $n\ge 2$.
    Therefore, for any subword $u$ of $w_n(x,y)$, we have
    \[ \frac{|u|}{d_S(u)} \le \frac{|\overline u|}{d_{\overline S}(\overline u)} \le \lambda_0 \]
    so $w_n$ is also $(\lambda_0,0)$-quasigeodesic in $G$ for $n\ge 2$.
\end{proof}

\begin{lemma} \label{lem:commutator}
  Let $L$ be a CFL containing all $(\lambda_0,0)$-quasigeodesics in $G$.
  Then there exists $N$ such that for all $n\ge N$,
  $L$ contains a word that has the commutator $[a^n,b^{n^2}]$ as a subword.
\end{lemma}
\begin{proof}
  This is the same proof as for $\Z^2$ in \thmref{thm:z2}. $N$ will have to 
  be chosen greater than the Ogden constant of $L$.
\end{proof}

\begin{definition}
  A basic commutator over a set $S$ is defined inductively as follows: 
  \begin{itemize}
    \item All elements of $S$ are basic commutators of weight 1.
    \item If $u,v$ are basic commutators of weights $r,s$ respectively, 
    then $[u,v]$ is a basic commutator of weight $r+s$
  \end{itemize}
\end{definition}

\begin{corollary}\label{cor:weight}
    If a basic commutator $c$ has weight $k$, then $c\in G_k$ 
    (here we abuse notation by identifying $c$ with its value in $G$).
\end{corollary}
\begin{proof}
    This follows by induction from the standard identity $[G_i, G_j] \le G_{i+j}$, 
    and the base case $G_1 = G$.
\end{proof}

The following theorem is an application of a symbolic rearrangement
process described by Hall in \cite{hall_edmonton_notes_2025}. 
The form we use is due to Struik \cite{struik}.

\begin{theorem}[Hall's collecting process \cite{struik}]\label{thm:collecting}
    Let $a,b$ be elements of a nilpotent group $G$.
    Then for any $m,n \ge 0$ we have the identity
    \[ b^m a^n = a^n b^m  [b,a]^{nm} [[b,a], a]^{\binom{n}{2}m} [[b,a], b]^{n\binom{m}{2}} \dots u_i^{f_i} \dots \]
    where the $u_i$ are fixed basic commutators in $\{a,b\}$ independent of $m,n$. 
    Further, if $u_i$ is a basic commutator involving $r$ $a$'s and $s$ $b$'s, then 
    \[ f_i = \sum_{j\le r, k \le s} c_{i,j,k} \binom{n}{j} \binom{m}{k} \]
    for some nonnegative integers $c_{i,j,k}$. In particular, the exponent of 
    such a $u_i$ is a polynomial in $n,m$ of highest degree term $Cn^{r}m^s$ for some $C>0$.
    \label{struik}
\end{theorem}

Note that the above expansion is necessarily a finite product since 
commutators of weight greater than the nilpotency class of $G$ vanish.
After rearranging, we get an expression for $[b^m, a^n]$, which is what we 
will use.

\begin{definition}
  Let $H$ be a subgroup of $G$, with finite generating sets $T,S$ respectively.
  The distortion function of $H$ in $G$ is 
  \[ \Delta_H^G(n) = \max \{ d_T(h) \mid h\in H, d_S(h) \le n \} \]
  where $d_T, d_S$ are the word metrics on $H,G$ respectively.
\end{definition}

\begin{definition}
  Let $g\in G$ be an element of infinite order.
  The \emph{weight} $v_G(g)$ of $g$ in $G$ is the largest $k$ such that 
  $\langle g \rangle \cap G_k \neq \{1\}$.
\end{definition}

The next theorem was given by Osin in \cite{osin} (we use a weaker 
version here: $\Omega$ can be replaced by $\Theta$).
\begin{theorem}[Subgroup distortion \cite{osin}]\label{thm:subgroupdistortion}
   Let $H$ be a subgroup of a finitely generated nilpotent group $G$, 
   and $H^0$ be the set of elements of infinite order in $H$.
   Then we have $\Delta_H^G(n) = \Omega(n^r)$, where
    \[ r = \max_{h\in H^0} \frac{v_G(h)}{v_H(h)}. \]
\end{theorem}

\begin{corollary}\label{cor:upperbound}
  Let $g\in G_k$. Then $d(g^n) = O(n^{1/k})$.
\end{corollary}
\begin{proof}
  If $g$ has infinite order, apply Osin's theorem to the subgroup $\langle g \rangle \cong \Z$ of $G$, noting that 
  $v_{\langle g \rangle}(g) = 1$ because $\langle g \rangle$ is abelian, and $v_G(g) \ge k$. This gives the 
  distortion function $\Delta_{\langle g \rangle}^G(n) =  \Omega(n^k)$, which implies $d(g^n) = O(n^{1/k})$.
  If $g$ has finite order, then $g^n$ takes only finitely many values, so $d(g^n)$ is bounded and clearly $O(n^{1/k})$ for 
  all $k\ge 1$.
\end{proof}

To motivate the next proof, consider the discrete Heisenberg group, which has nilpotency class $2$ and may be defined as 
$ H = \langle x, y \mid [[x,y],x] = [[x,y],y] = 1\rangle$.
Let $t = [x,y]$; then using that $t,x$ and $t,y$ commute in $H$, we have $[x^{d_1},y^{d_2}] = t^{d_1d_2}$ for all $d_1, d_2\ge 0$.
It is helpful to think of $[\overline{x}^{d_1},\overline{y}^{d_2}]$ tracing out a rectangle in $H_{ab}\cong \langle \overline{x}, \overline{y}\rangle \cong \Z^2$, and $d_1d_2$ as the area of the rectangle. \corref{cor:upperbound} is now apparent: the displacement of $t^n$ is proportional to $\sqrt{n}$. 
In $H_{ab}$, the image of the word $[x^n,y^{n^2}]$
traces a $n\times n^2$ rectangle, covering an area of 
$n^3$. Such a rectangular path is \emph{inefficient} 
because we can cover the same area with a lower perimeter, 
by using a square path instead. This is why $[x^n,y^{n^2}]$
is an arbitrarily bad quasigeodesic in $H$ for large $n$. While the analogy to the area enclosed by a rectangle does not hold in generality, \thmref{thm:collecting} approximates this well enough for the following upper bound on the displacement of $[a^n,b^{n^2}]$ in $G$.

\begin{lemma}\label{lem:upperbound}
    We have $d([a^n,b^{n^2}]) = o(n^2)$.
\end{lemma}
\begin{proof}
  From \thmref{thm:collecting}, we have the identity 
  \[ [b^{n^2}, a^n] = [b,a]^{n^3} [[b,a],a]^{\binom{n}{2} n^2} [[b,a],b]^{n\binom{n^2}{2}} \dots u_i^{f_i} \dots   \]
  where the $u_i$ are basic commutators in $\{a,b\}$, and the $f_i$ are polynomials in $n$. 
  Further, if $u_i$ has weight $W_i = r_i+s_i$, with $r_i$ $a$'s and $s_i$ $b$'s appearing in it,
  then $f_i$ has highest degree term of the form $C n^{r_i} (n^2)^{s_i} = C n^{r_i+2s_i}$ for some $C>0$. 
  Since $r_i + 2s_i \le 1 + 2(W_i-1) = 2W_i - 1$, we have $f_i = O(n^{2W_i-1})$. 
  Therefore, because $u_i \in G_{W_i}$ from \corref{cor:weight}, we may apply \corref{cor:upperbound} to get
  \[ d(u_i^{f_i}) = O(f_i^{1/W_i}) = O(n^{2 - 1/W_i}) . \] 
    Now by the triangle inequality, 
  \[ d([a^n,b^{n^2}]) \le \sum_{i} d(u_i^{f_i}) =O(n^{2 - 1/\max_i W_i}) = o(n^2)  \]
  since we are dealing with a finite sum.
\end{proof}

We put together three facts to complete the proof: 
the uniform quasigeodesicity of $w_n$
from \lemref{lem:uniform}, the upper bound on $d([a^n,b^{n^2}])$ from \lemref{lem:upperbound},
and the fact that the word length of 
$[a^n,b^{n^2}]$ is $\Theta(n^2)$.

\begin{theorem}\label{thm:nilpotent}
  Let $G$ be a finitely generated nilpotent group that is not virtually cyclic.
  Then $G$ is chaotic.
\end{theorem}
\begin{proof}
  From \lemref{lem:uniform} there exists $\lambda_0$ such that 
  $w_n$ is $(\lambda_0, 0)$-quasigeodesic for all $n\ge 2$. 
  We will show that $G$ is chaotic at $(\lambda_0,0)$ (with respect to the generating set $S$).
  Let $L$ be a CFL containing all $(\lambda_0,0)$-quasigeodesics in $G$.
  By \lemref{lem:commutator}, for all sufficiently large $n$, 
  $L$ contains a word with subword $[a^n,b^{n^2}]$.
  But $d_S([a^n,b^{n^2}]) = o(n^2)$ from \lemref{lem:upperbound} and $|[a^n,b^{n^2}]| = 2n + 2n^2 = \Theta (n^2)$.
  Therefore, as $n\to\infty$,
  \[ \frac{|[a^n,b^{n^2}]|}{d_S([a^n,b^{n^2}])} \to \infty, \]
  so for any fixed $(\lambda,\epsilon)$,
  by taking $n$ large enough, we can show that $L$ contains a word that is not
  $(\lambda,\epsilon)$-quasigeodesic.
\end{proof}

\begin{corollary}
    Let $G$ be a group of polynomial growth,
    with context-free rational quasigeodesic languages. 
    Then $G$ is virtually cyclic.
\end{corollary}
\begin{proof}
    Note that \thmref{thm:nilpotent} extends to non-virtually-cyclic \emph{virtually} nilpotent groups by \thmref{thm:inherit}.
    Then the corollary follows by taking the contrapositive of \thmref{thm:nilpotent}, and 
    applying Gromov's theorem on polynomial growth \cite{Gromov1981}.
\end{proof}
The hyperbolic nilpotent groups are exactly the virtually cyclic ones, so the $\Q$\textbf{CF} property is a characterization of hyperbolicity at least within the class of polynomial growth groups.

\section{Baumslag--Solitar groups}
The goal is to show that $BS(m,n)$ is chaotic for all $m, n$. 
We will use a pumping argument based on \corref{cor:ogden} as before,
but with a more involved construction of the required quasigeodesic family of words. 

\begin{definition}
    The Baumslag--Solitar group $BS(m,n)$ is defined by the presentation
    \[ BS(m,n) = \langle a,t \mid ta^m t^{-1} = a^n \rangle . \]
    We will always consider it with the generating set $\{a,t\}$.
\end{definition}

If \(m=\pm n\), then \(BS(m,n)\) contains an undistorted subgroup
isomorphic to \(\mathbb Z^2\), namely \(\langle t^2,a^m\rangle\), so these
cases follow from \thmref{thm:z2} and
\thmref{thm:inherit}. Using the isomorphisms
\(BS(m,n)\cong BS(-m,-n)\), induced by \(a\mapsto a^{-1}\), and
\(BS(m,n)\cong BS(n,m)\), induced by \(t\mapsto t^{-1}\), we may assume
that \(m\ge 1\) and \(|n|>m\). We focus on the case when $n$ is also positive, and discuss how to 
adapt the construction for negative $n$ later.
So $n > m \ge 1$; fix such $m,n$ and let $G = BS(m,n)$.

\begin{definition}\label{def:BSwords}
    Define the sequence $(d_i)$ recursively by $d_0 = 0$, $d_1 = 1$, and
    $d_{i+1} = \lceil (n/m) d_i \rceil$ for $i\ge 1$.  
    Now let $y_i = md_i$ and $z_i = nd_i$ for $i\ge 0$, 
    and $x_i = y_{i+1} - z_i$ for $i\ge 0$.
    In general, the following is true for $i\ge 1$: $y_{i+1}$ is the smallest multiple of $m$
    that is at least $z_i$, and so $0 \le x_i < m$.
    For $q\ge 1$, let
    \begin{align*}
        u_q &= a^{x_0}t^{-1}a^{x_1}t^{-1}\dots t^{-1}a^{x_{q-1}}t^{-1} \\
        v_q &= a^{-x_0}t^{-1}a^{-x_1}t^{-1}\dots t^{-1}a^{-x_{q-1}}t^{-1} 
    \end{align*}
    and define
    \[ w_q = (a^{-1} t^{2q}v_{2q} a v_{2q}^{-1})t^{-q} (a t^{2q}u_{2q} a^{-1}u_{2q}^{-1}) . \]
\end{definition}

For reference, a special case of the sequence $(d_i)$ for $m = 2$, 
$n = 3$ appears in OEIS as A061419 \cite{OEIS-A061419}.
The choice of $u_q, v_q$ is motivated by the following corollary: 
it ensures that the words $t^{2q}u_{2q}$ and $t^{2q}v_{2q}$ trace a path
that stays in a single `sheet' of the Cayley graph: see \figref{fig:BSworddraw}.

\begin{corollary}\label{cor:upvp}
    For $q\ge 1$, we have $t^{q}u_q = a^{z_q}$ and $t^{q}v_q = a^{-z_q}$ in $G$.
\end{corollary}
\begin{proof}
    We prove the statement for $u_q$; the proof for $v_q$ is similar.
    We use induction on $q$. For the base case $q = 1$, we have $tu_1 = t a^m t^{-1} = a^n = a^{z_1}$. 
    Now suppose $t^{q}u_q = a^{z_q}$ for some $q\ge 1$.
    Then 
    \begin{align*} 
    t^{q+1}u_{q+1} &= t(t^{q}u_q)a^{x_q}t^{-1} \\
    &= t a^{z_q + x_q} t^{-1} \\
    &= t a^{y_{q+1}} t^{-1} \\
    &=  a^{ny_{q+1}/m} \\
    &= a^{z_{q+1}}.\qedhere
    \end{align*}
\end{proof}

\begin{figure}[h]
\begin{tikzpicture}[
  >=Latex,
  every node/.style={font=\small},
  aedge/.style={line width=.9pt,->},
  aedgedash/.style={line width=.9pt,->,dash pattern=on 3pt off 2pt},
  tedge/.style={line width=.9pt,->},
  thinaux/.style={line width=.6pt,draw=black!55},
]

\def\H{1.25}   
\def\L{8.0}    
\def\r{0.62}   

\coordinate (L0) at (0,0);               
\coordinate (R0) at (\L,0);              
\coordinate (M0) at ({\r*\L},0);         

\coordinate (L1) at (0,\H);              
\coordinate (R1) at (\L,\H);             
\coordinate (M1) at ({\r*\L},\H);        

\coordinate (L2) at (0,-\H);              
\coordinate (R2) at (\L,-\H);             
\coordinate (M2) at ({\r*\L},-\H);        

\draw[blue,tedge] (L0) -- (L1) node[midway,left=2pt] {$t$};
\draw[blue,tedge] (L2) -- (L0) node[midway,left=2pt] {$t$};
\draw[blue,tedge] (R0) -- (R2) node[midway,right=2pt] {$t^{-1}$};
\draw[blue,tedge] (M1) -- (M0) node[midway,right=2pt] {$t^{-1}$};

\draw[red,aedge] (L1) -- (M1)
  node[midway,above=3pt] {$a^{x_{0}}=a^{m}$};

\draw[aedgedash] (L0) -- (M0)
  node[midway,above=3pt] {$a^{z_1} = a^{n}$};
\draw[red,aedge] (M0) -- (R0)
  node[midway,above=3pt] {$a^{x_{1}}$};
\draw[aedgedash] (L2) -- (R2)
  node[midway,above=3pt] {$a^{z_2}$};

\end{tikzpicture}
\caption{The word $t^2u_2 = t^2 a^{x_0} t^{-1} a^{x_1} t^{-1}$ traced in the Cayley graph of $BS(m,n)$. 
$x_1$ is, by definition, the least integer such that $y_2 = z_1+x_1$ is a multiple of $m$.}
\label{fig:BSworddraw}
\end{figure}
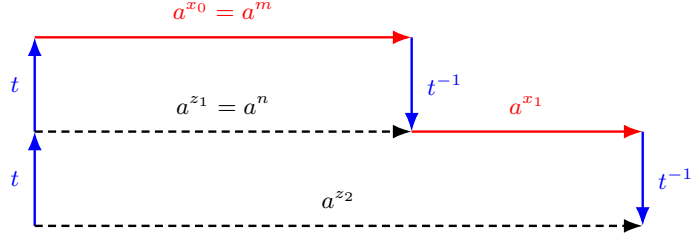

\begin{corollary}\label{cor:displacement}
    We have $t^{-q} (a t^{2q}u_{2q} a^{-1}u_{2q}^{-1}) = t^{q}$ and
    $(a^{-1} t^{2q}v_{2q} a v_{2q}^{-1})t^{-q} = t^q$ in $G$.
\end{corollary}
\begin{proof}
    We prove the first statement; the second is symmetric.
    Using \corref{cor:upvp} to substitute for $t^{2q}u_{2q}$, we have 
    \begin{align*}
        t^{-q} (a t^{2q}u_{2q} a^{-1}u_{2q}^{-1}) &= t^{-q} (at^{2q}u_{2q}a^{-1}(t^{2q}u_{2q})^{-1})t^{2q}  \\
        &= t^{-q} (a a^{z_{2q}} a^{-1} a^{-z_{2q}})t^{2q} \\
        &= t^{q}. \qedhere
    \end{align*} 
\end{proof}

Intuitively, we would like to repeat a construction like that 
in \figref{fig:spiral}, but we run into problems due to the exponential 
collapse of $\langle a \rangle$-distances in the Cayley graph of $BS(m,n)$ 
as we climb to higher $t$-levels of the graph (so one cannot naively construct the same kind of spiral). We get around this by using 
more than one `sheet' of the Cayley graph: one sheet to travel along a subword equal 
to a large power of $a$ (the subword $t^{2q}u_{2q}$), and another subword ($a^{-1}u_{2q}^{-1}$)
to complete the analog of the spiral arm. See \figref{fig:wordtrace}.

\begin{figure}[h]
    \begin{tikzpicture}[line cap=round, line join=round, >=Latex, semithick,
                        shorten >=1pt, shorten <=1pt, every node/.style={font=\small}]
    
    \tikzset{slightwave/.style={
    decorate,
    decoration={snake, amplitude=0.07cm, segment length=12mm, pre length=0pt, post length=0pt}
    }}

    \def\xa{0}     
    \def\xb{1}     
    \def\ybot{0}   
    \def\yl{2.5}   
    \def\yu{5}     
    \def\yend{0}   
    \def\off{0.3}  

    \pgfmathsetmacro{\dx}{\yu - \yend}

    \coordinate (Lbot) at (\xa,\ybot);
    \coordinate (Ltop) at (\xa,\yl);
    \coordinate (Rbot) at (\xb,\ybot);
    \coordinate (Rtop) at (\xb,\yu);
    \coordinate (Uend) at (\xb+\dx,\yend);
    \coordinate (Lend) at (\xb+\dx-1,\yend); 

    \coordinate (BBot) at (0, -2.5);
    \coordinate (Bleft) at (-1, -2.5);
    \coordinate (Btop) at (-1, 2.5);
    \coordinate (BFleft) at (-6, -2.5);
    \coordinate (BFleftoff) at (-5, -2.5);

    \draw[<-] (Lbot) -- (Ltop) node[midway,left=1pt] {$t^{-q}$};
    \draw[->] (Rbot) -- (Rtop) node[midway,right=1pt] {$t^{2q}$};

    \draw[->] (Bleft) -- (Btop) node[midway,left=1pt] {$t^{2q}$};

    \draw[dashed,->] (\xa,\yu) -- (\xa,\yl)
        node[midway,left=3pt] {$t^{-q}$};
    
    \draw[dashed,->] (0,0) -- (BBot)
        node[midway,left=3pt] {$t^{-q}$};

    \draw[->,slightwave] (Rtop) -- (Uend)
        node[pos=0.58,above right=1pt] {$u_{2q}$};

    \draw[<-,slightwave] (\xa,\yu) -- (Lend)
        node[pos=0.55,below left=1pt] {$u_{2q}^{-1}$};

    \draw[->,slightwave] (Btop) -- (BFleft)
        node[pos=0.58,above left=1pt] {$v_{2q}$};

    \draw[->,slightwave] (BFleftoff) -- (Ltop)
        node[pos=0.55,below =10pt] {$v_{2q}^{-1}$};

    \draw[->] (Lbot) -- (Rbot) node[midway,below=3pt] {$a$};
    \draw[<-] (Lend) -- (Uend) node[midway,below=3pt] {$a^{-1}$};

    \draw[->] (BBot) -- (Bleft) node[midway,below=3pt] {$a^{-1}$};
    \draw[->] (BFleft) -- (BFleftoff) node[midway,below=3pt] {$a$};

    \end{tikzpicture}
    \caption{The word $w_q$ traced 
    on a projection of the Cayley graph of $BS(m,n)$. It appears to self-intersect only
    because of the 2D projection; the subpaths in question always lie in different sheets of the Cayley graph.}
    \label{fig:wordtrace}
\end{figure}
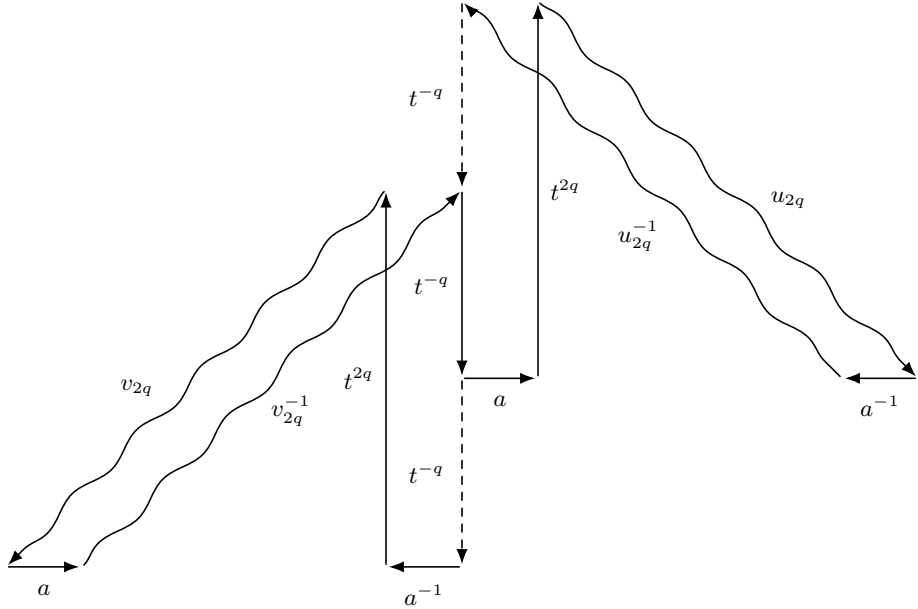

The main difficulty now is in proving that the words $w_q$ are actually uniformly quasigeodesic.

\begin{lemma}\label{lem:qg}
    There exist $(\lambda_0, \epsilon_0)$ such that for all $q\ge 1$,
    the word $w_q$ is $(\lambda_0, \epsilon_0)$-quasigeodesic in $BS(m,n)$.
\end{lemma}

We will prove this lemma by calculating a linear lower bound on the displacement of 
all subwords of $w_q$, of the form $d(u) \ge C|u| - D$, for uniform constants $C,D > 0$. 
To this end we rely on a normal form for words in $BS(m,n)$ introduced by 
Burillo and Elder in \cite{BurilloElder2014}, and a corresponding lower bound on displacement for words in 
the normal form. This is similar to the general normal form theorem for HNN extensions 
\cite{LyndonSchupp1977}.

\begin{lemma}[\cite{BurilloElder2014}]\label{lem:normalform}
    Any element $g$ of $BS(m,n)$ can be written uniquely as 
    $g = w(a,t)a^N$ where 
    \[ w(a,t) = \{ t, at, a^2t, \dots, a^{n-1}t, t^{-1}, at^{-1}, a^2t^{-1}, \dots a^{m-1}t^{-1} \}* \] 
    and $g$ is freely reduced.
\end{lemma}
An arbitrary word in $BS(m,n)$ can be reduced to this normal form 
by performing the following rewritings:
\begin{itemize}
    \item freely reduce the word by cancelling $tt^{-1}, aa^{-1}$;
    \item replace $a^{r\pm dn}t$ by $a^rta^{\pm dm}$ where $0 \le r < n$;
    \item replace $a^{s\pm dm}t^{-1}$ by $a^st^{-1}a^{\pm dn}$ where $0 \le s < m$.
\end{itemize}

\begin{lemma}[\cite{BurilloElder2014}]\label{lem:distbound}
    There exist constants $C,D > 0$ such that for every $g\in BS(m,n)$ written in 
    the normal form $w(a,t)a^N$, we have 
    \[ d(g)\ge C(|w| + \log_{n/m} (|N|+1)) - D . \]
    In particular, we can take $C = 1/(n+1)$ and $D = \log_{n/m}(\tfrac{mn}{n-m})$.
\end{lemma}

\begin{corollary}\label{cor:distbound}
    Let $g = a^{N'}w(a,t)$ where $w(a,t)$ has the form described in \lemref{lem:normalform} and 
    the letters $t, t^{-1}$ appear $T$ times in $w(a,t)$. Then
    \[ d(g) \ge C_1|w| + C_2 \log \left(1 + \left( \frac{m}{n}\right)^T |N'| \right) - D_1 \]
    for some constants $C_1, C_2, D_1 > 0$. If $N' \ge 0$, these are given by
    \begin{align*}
        C_1 &:= \frac{1}{n(n+1)},  \qquad
        C_2 := \frac{1}{(n+1)\log(n/m)}, \\ 
        D_1 &:= \log_{n/m}(\tfrac{nm}{n-m}) + \frac{1}{n+1}\log_{n/m}(1 + \tfrac{n^2}{n-m}).
    \end{align*}
\end{corollary}
\begin{proof}
    We give a proof for the case $N' \ge 0$. 
    The proof for $N'<0$ is identical, save for some extra bookkeeping 
    of constants. We can convert $g = a^{N'}w(a,t)$ to the normal form from \lemref{lem:normalform}
    by moving the large power $a^{N'}$ to the right via a sequence of rewrites. 
    Consider the generic rewrites $a^Xt \to a^rta^{X'}$ and $a^Xt^{-1} \to a^sta^{X'}$. 
    In the first case we have $X = r+dn$ and $X' = dm$ for some $r,d$,
    which gives $X' \ge (X-n)(m/n) = \tfrac{Xm}{n}-m$, and similarly
    $X'\ge (X-m)(n/m) = \tfrac{Xn}{m}-n$ for the second case. This gives 
    \[ X'\ge \min(\tfrac{Xn}{m}-n, \tfrac{Xm}{n}-m) \ge \tfrac{Xm}{n}-n . \] 

    Note that $a^{N'}$ may interact with intermediate powers of $a$ in the word 
    $w(a,t)$ as it is pushed to the right; however, since these intermediate powers 
    are always positive by the definition of the normal form, the lower bound 
    for $X'$ still holds (here we use that $N'\ge 0$).

    Thus, once $a^{N'}$ has been pushed through $T$ $t$'s, we have
    $a^{N'}w(a,t) \to w'(a,t)a^{N''}$. After summing the geometric series resulting 
    from iterating the lower bound for $X'$, we get
    \begin{align*}
        N'' &\ge  N'\left(\frac{m}{n}\right)^T- \frac{n^2}{n-m}\left(1-\left(\frac{m}{n}\right)\right)^T \\ 
        &\ge N'\left(\frac{m}{n}\right)^T- \frac{n^2}{n-m}.
    \end{align*}
    Also, trivially $N''\ge 0$.
    Now applying \lemref{lem:distbound} to $ w'(a,t)a^{N''}$, we have 
    \[ d(g) = d(w'(a,t)a^{N''}) \ge \frac{1}{n+1}\left( |w'|+ \log_{n/m}(1+N'')\right) - \log_{n/m}(\tfrac{mn}{n-m}) . \]
    Since at most $n$ $a$'s follow a $t$ or $t^{-1}$ in $w'(a,t)$, 
    $|w'|\ge T \ge |w|/n$. We also have $N'' \ge \max(0, N'\left(\frac{m}{n}\right)^T- \frac{n^2}{n-m})$, 
    and using the concavity of the logarithm we get 
    \begin{align*}
        d(g) \ge& \, \frac{1}{n(n+1)}|w| + \frac{1}{n+1}\log_{n/m}\left( 1 + \left(\frac{m}{n}\right)^T N'\right) \\
        &- \frac{1}{n+1}\log_{n/m}(1 + \tfrac{n^2}{n-m}) - \log_{n/m}(\tfrac{nm}{n-m}).\qedhere
    \end{align*} 
\end{proof}

\begin{lemma}\label{lem:constants}
    There are constants $C,D > 0$ such that for all subwords $u$ of $w_q$, 
    we have $d(u) \ge C|u| - D$.
\end{lemma}
\begin{proof}
    Recall that
    \[ w_q = (a^{-1} t^{2q}v_{2q} a v_{2q}^{-1})t^{-q} (a t^{2q}u_{2q} a^{-1}u_{2q}^{-1}) . \]
    Fix $q$, and write $L:= a^{-1} t^{2q}v_{2q} a v_{2q}^{-1}$, $M:= t^{-q}$, $R:= a t^{2q}u_{2q} a^{-1}u_{2q}^{-1}$. 
    Let $u$ be a subword of $w_q$; the proof is by cases on the position of $u$: 
    \begin{enumerate}
        \item $u$ lies entirely within one of $L,M,R$.
        \item $u$ crosses one boundary, either $L|M$ or $M|R$.
        \item $u$ crosses both boundaries $L|M$ and $M|R$.
    \end{enumerate}
    The proof is similar in all cases, essentially 
    by rewriting $u$ in normal form (as in \lemref{lem:normalform}) and applying \lemref{lem:distbound} or \corref{cor:distbound}
    to get the desired linear lower bound on $d(u)$. This might involve 
    considering several subcases. We give the details for case (1) only.

    \emph{Case (1):} if $u$ lies in $M$, then $u$ is a power of $t$ and actually geodesic, so any 
    $0 < C\le 1$ and $D>0$ will work. Suppose $u$ lies in $L$ (the case $u$ in $R$ is symmetric). Observe that the subword
    $t^{2q}v_{2q}$ is simply a large power of $a$ when written in normal form (see \corref{cor:upvp}). 
    We analyze \textit{how much} of this subword is contained in $u$, and rewrite $u$ in normal 
    form accordingly. Define $S[i]$, for $0\le i \le 2q$, to be the subword of $t^{2q}v_{2q}$ that is equal to $t^iv_i$.
    Let $j$ be the largest integer such that $S[j]$ is a subword of $u$. 

    Then either $u = a^{-1} t^{2q} v_{j}a^k$ for some $0 \le k < m$ or $u = t^{j} v_{j}X$ for some $X$. 
    \begin{itemize}
        \item In the first case, we have $u = a^{-1} t^{2q-j} (t^jv_{j})a^k = a^{-1} t^{2q-j} a^{-z_j+k}$ in $G$ by \corref{cor:upvp}.
        In normal form, this is $u = a^{n-1}ta^mt^{2q-j}a^{-z_j+k}$. Then applying \lemref{lem:distbound} gives 
        \[ d(u)\ge C((n+m+ 2q-j) + \log_{n/m} (|z_j-k| + 1)) - D. \]
        It's clear from the definition of $z_j$ that $z_j \ge (n/m)^{j}$ for $j\ge 1$, so 
        $\log_{n/m} (|z_j-k| + 1) \ge j - m$. Therefore, $d(u)\ge C(n+2q) - D$. 
        Since $|u| \le 1 + 2q + jm + k \le 1 + 2q+2qm + m = (2q+1)(m+1)$, there are suitable 
        constants $C', D'$ such that $d(u) \ge C'|u| - D'$.

        \item In the second case, $u = t^{j} v_{j}X$, and in fact this $X$ can be written 
        as some $w(a,t)$ satisfying the conditions of \lemref{lem:normalform}, because
        \[ v_{2q} a v_{2q}^{-1} \in \{ t, at, a^2t, \dots, a^{n-1}t, t^{-1}, at^{-1}, a^2t^{-1}, \dots a^{m-1}t^{-1} \}^*  \]
        and $X$ is a subword of this. Then in $G$, we have the equality $u = a^{-z_j} w(a,t)$ by \corref{cor:upvp}.
        Now by \corref{cor:distbound} we have 
        \[ d(u) \ge C_1|w| + C_2 \log \left(1 + \left( \frac{m}{n}\right)^T z_j \right) - D_1 \]
        for some constants $C_1, C_2, D_1 > 0$,  and where $T$ is the number of $t, t^{-1}$ letters in $w(a,t)$. 
        From $z_j \ge (n/m)^{j}$ and $|w| \ge T$, we have 
        \[ d(u)\ge C_1|w| + C_2(\max(0,j-T)) - D_1 . \]
        If $j \le 2|w|$, then $d(u) \ge C_1|w| - D_1 \ge C'|u| - D'$
        for some $C', D' > 0$ since $|u| \le |w| + jm \le |w| + 2|w|m = (2m+1)|w|$.
        Otherwise, if $j > 2|w|$, then in particular $j > 2T$ so 
        $d(u) \ge  C_2(\max(0,j-T)) - D_1 \ge (C_2/2)j - D_1$. 
        But also $|u| \le jm + |w| \le jm + j/2 = j(m+1/2)$, so again we can find constants $C', D'$ such that
        $d(u) \ge (C_2/2)j - D_1 \ge C'|u| - D'$.
    \end{itemize}
    
    By taking the minimum $C'$ and maximum $D'$ over all (finitely many) subcases, we
    get uniform $C,D > 0$ that works for case $(1)$.
\end{proof}

\begin{proof}[Proof of \lemref{lem:qg}]
    Using \lemref{lem:constants}, find $C,D$ such that 
    for all subwords $u$ of $w_q$, we have $d(u) \ge C|u| - D$.
    Then $|u| \le \tfrac{1}{C}d(u) + \tfrac{D}{C}$, so we 
    can take $(\lambda_0, \epsilon_0) = (\tfrac{1}{C}, \tfrac{D}{C})$. \qedhere
\end{proof}

\begin{theorem}
  $BS(m,n)$ for $n > m\ge 1$ is chaotic.
\end{theorem}
\begin{proof}
We use a very similar pumping argument to \thmref{thm:z2}.
Let $L$ be a CFL containing all $(\lambda_0,0)$-quasigeodesics in $BS(m,n)$,
with Ogden constant $p$, and let $\lambda \ge 1$, $\epsilon \ge 0$ be arbitrary.
Let $q = \lceil\max(p, \epsilon)\rceil$ and consider the word $w_q$ as defined above. Apply \corref{cor:ogden}, 
with $\alpha = (a^{-1} t^{2q}v_{2q} a v_{2q}^{-1})$, $\beta = t^{-q}$, and $\gamma = (a t^{2q}u_{2q} a^{-1}u_{2q}^{-1})$. 
This gives a decomposition $w_q = uxzyv$, where either $x$ or $y$ is fully contained in
in $\beta = t^{-q}$. Without loss of generality suppose $x$ is fully contained in $t^{-q}$, so that $x = t^{-k}$ for some $1 \le k \le p$.
After pumping,
\[ w_{q,N} = ux^Nzy^Nv = (a^{-1} t^{2q}v_{2q} a v_{2q}^{-1})t^{-q-kN}X_N \] 
for some words $X_N$; the prefix $(a^{-1} t^{2q}v_{2q} a v_{2q}^{-1})$ is not affected by the pumping. Choose 
$N \ge q/k$, so that $w_{q,N}$ has a prefix $w' = (a^{-1} t^{2q}v_{2q} a v_q^{-1})t^{-2q}$. But this subword 
is not $(\lambda, \epsilon)$-quasigeodesic, because $|w'| \ge 4mq + 7q + 2$ and $d(w') = 0$, so in particular 
$|w'| \ge 7\epsilon > \lambda \cdot 0 + \epsilon$. Since $w_{q,N} \in L$, we have shown that 
$L$ contains a word that is not $(\lambda, \epsilon)$-quasigeodesic.
\end{proof}

The case when $n$ is negative presents minor complications.
For example, the word $at^{-1}a$ is geodesic in $BS(1,2)$ but not in $BS(1,-2)$.
In analogy with the Klein bottle group $BS(1,-1)$, in some sense directions are reversed 
at adjacent $t$-levels. 
The same proofs and ideas work in this case, after slightly modifying the definition of $w_q$.
For example, replacing all $t$'s with $t^2$ in \defref{def:BSwords} is enough. Alternately, 
some signs in \defref{def:BSwords} need to be swapped.

\begin{theorem}\label{thm:BS}
    All Baumslag--Solitar groups are chaotic.
    In particular, they do not have the $\Q \emph{\bf{CF}}$ property.
\end{theorem}

\section{Conclusion}

We can combine the main feature of the chaoticity property (\thmref{thm:inherit})
with \thmref{thm:nilpotent} and \thmref{thm:BS} to lift to overgroups 
in which nilpotent or BS groups are QI-embedded. 
\begin{theorem}\label{thm:main}
    Let $G$ be a finitely generated group, 
    containing an undistorted subgroup isomorphic to a
    non-virtually-cyclic nilpotent group, or a Baumslag--Solitar group.
    Then $G$ is chaotic, and in particular, is not $\Q \emph{\bf{CF}}$.
\end{theorem}

The techniques in this paper do not extend well to the general case of 
\conjref{conj:main}, 
as we have made explicit use of particular structural properties of nilpotent and Baumslag--Solitar groups.
The crux of the problem lies in the fact that there seems to be little to link the algebraic and geometric structure of an arbitrary finitely generated group to the combinatorial structure of its quasigeodesic words. 

Lastly, arguments in this paper say nothing about the minimal
$(\lambda_0, \epsilon_0)$ for which groups fail to have context-free $(\lambda_0, \epsilon_0)$-quasigeodesic 
languages, because the $\Q\bf{CF}$ property 
is too coarse for this question.
We therefore end with the following more ambitious conjecture.

\begin{conjecture}
    There exists a universal constant $\lambda_0$ such that if $G$ is a non-hyperbolic finitely generated group with generating set $S$,
    then for all
    $\lambda > \lambda_0$, the 
    $(\lambda,0)$-quasigeodesics in $\Gamma(G,S)$ do not 
    form a context-free language.
\end{conjecture}

\bibliographystyle{alpha}  
\bibliography{references}

\end{document}